\documentclass{amsart}

\usepackage{amsthm}
\usepackage[leqno]{amsmath}
\usepackage[all]{xy} \SelectTips{eu}{}
\usepackage{latexsym,amsfonts,amssymb}
\usepackage{hyperref}
\usepackage{mathptmx}
\usepackage{nicefrac}
\usepackage{stmaryrd}

\newcommand{\numberseries}{\bfseries}   

\newlength{\thmtopspace}                
\newlength{\thmbotspace}                
\newlength{\thmheadspace}               
\newlength{\thmindent}                  

\newtheoremstyle{fixed bf head,slanted body}
                {\thmtopspace}{\thmbotspace}{\slshape}
                {\thmindent}{\bfseries}{.}{\thmheadspace}
                {{\numberseries \thmnumber{#2\;}}\thmname{#1}\thmnote{ (#3)}}

\newtheoremstyle{variable bf head,slanted body}
                {\thmtopspace}{\thmbotspace}{\slshape}
                {\thmindent}{\bfseries}{.}{\thmheadspace}
                {{\numberseries \thmnumber{#2\;}}\thmname{#1}\thmnote{ #3}}

\newtheoremstyle{fixed bf head,upright body}
                {\thmtopspace}{\thmbotspace}{\upshape}
                {\thmindent}{\bfseries}{.}{\thmheadspace}
                {{\numberseries \thmnumber{#2\;}}\thmname{#1}\thmnote{ (#3)}}

\newtheoremstyle{numbered paragraph}
                {\thmtopspace}{\thmbotspace}{\upshape}
                {\thmindent}{\upshape}{}{\thmheadspace}
                {{\numberseries \thmnumber{#2.}}}

\theoremstyle{fixed bf head,slanted body}
\newtheorem{res}{}[section]
\newtheorem{thm}[res]{Theorem}          \newtheorem*{thm*}{Theorem}
      \newtheorem*{prp*}{Proposition}
        \newtheorem*{cor*}{Corollary}
\newtheorem{lem}[res]{Lemma}            \newtheorem*{lem*}{Lemma}

\theoremstyle{variable bf head,slanted body}
     \newtheorem*{introthm*}{Theorem}

\theoremstyle{fixed bf head,upright body}
            \newtheorem*{stp*}{Setup}
\newtheorem{dfn}[res]{Definition}       \newtheorem*{dfn*}{Definition}
     \newtheorem*{con*}{Construction}
      \newtheorem*{obs*}{Observation}
\newtheorem{rmk}[res]{Remark}           \newtheorem*{rmk*}{Remark}
\newtheorem{exa}[res]{Example}          \newtheorem*{exa*}{Example}
         \newtheorem*{qst*}{Question}

\theoremstyle{numbered paragraph}
\newtheorem{ipg}[res]{}

\newlength{\thmlistleft}        
\newlength{\thmlistright}       
\newlength{\thmlistpartopsep}   
\newlength{\thmlisttopsep}      
\newlength{\thmlistparsep}      
\newlength{\thmlistitemsep}     

\newcounter{eqc} 
  {\end{list}}%

\newcounter{prt}
\newenvironment{prt}{\begin{list}{\upshape (\alph{prt})}%
    {\usecounter{prt}%
      \setlength{\leftmargin}{\thmlistleft}%
      \setlength{\labelwidth}{\thmlistleft}%
      \setlength{\rightmargin}{\thmlistright}%
      \setlength{\partopsep}{\thmlistpartopsep}%
      \setlength{\topsep}{\thmlisttopsep}%
      \setlength{\parsep}{\thmlistparsep}%
      \setlength{\itemsep}{\thmlistitemsep}}}%
  {\end{list}}%

\newcommand{\prtlbl}[1]{{\upshape(#1)}}

\newcommand{\pgref}[1]{\ref{#1}}
\newcommand{\thmref}[2][Theorem~]{#1\pgref{thm:#2}}

\newcommand{\lemref}[2][Lemma~]{#1\pgref{lem:#2}}

\newcommand{\dfnref}[2][Definition~]{#1\pgref{dfn:#2}}

\newcommand{\rmkref}[2][Remark~]{#1\pgref{rmk:#2}}
\newcommand{\secref}[2][Section~]{#1\ref{sec:#2}}

\renewcommand{\eqref}[1]{(\pgref{eq:#1})}

\makeatletter
\def\@nobreak@#1{\mathchoice%
  {\nobreakdef@\displaystyle\f@size{#1}}%
  {\nobreakdef@\nobreakstyle\tf@size{\firstchoice@false #1}}%
  {\nobreakdef@\nobreakstyle\sf@size{\firstchoice@false #1}}%
  {\nobreakdef@\nobreakstyle\ssf@size{\firstchoice@false #1}}%
  \check@mathfonts}%
\def\nobreakdef@#1#2#3{\hbox{{%
                    \everymath{#1}%
                    \let\f@size#2\selectfont%
                    #3}}}%
\makeatother

\DeclareSymbolFont{usualmathcal}{OMS}{cmsy}{m}{n}
\DeclareSymbolFontAlphabet{\mathcal}{usualmathcal}

\DeclareSymbolFont{letters}{OML}{txmi}{m}{it}

\DeclareMathSymbol{\alpha}{\mathord}{letters}{"0B}
\DeclareMathSymbol{\beta}{\mathord}{letters}{"0C}
\DeclareMathSymbol{\gamma}{\mathord}{letters}{"0D}
\DeclareMathSymbol{\delta}{\mathord}{letters}{"0E}
\DeclareMathSymbol{\epsilon}{\mathord}{letters}{"0F}
\DeclareMathSymbol{\zeta}{\mathord}{letters}{"10}
\DeclareMathSymbol{\eta}{\mathord}{letters}{"11}
\DeclareMathSymbol{\theta}{\mathord}{letters}{"12}
\DeclareMathSymbol{\iota}{\mathord}{letters}{"13}
\DeclareMathSymbol{\kappa}{\mathord}{letters}{"14}
\DeclareMathSymbol{\lambda}{\mathord}{letters}{"15}
\DeclareMathSymbol{\mu}{\mathord}{letters}{"16}
\DeclareMathSymbol{\nu}{\mathord}{letters}{"17}
\DeclareMathSymbol{\xi}{\mathord}{letters}{"18}
\DeclareMathSymbol{\pi}{\mathord}{letters}{"19}
\DeclareMathSymbol{\rho}{\mathord}{letters}{"1A}
\DeclareMathSymbol{\sigma}{\mathord}{letters}{"1B}
\DeclareMathSymbol{\tau}{\mathord}{letters}{"1C}
\DeclareMathSymbol{\upsilon}{\mathord}{letters}{"1D}
\DeclareMathSymbol{\phi}{\mathord}{letters}{"1E}
\DeclareMathSymbol{\chi}{\mathord}{letters}{"1F}
\DeclareMathSymbol{\psi}{\mathord}{letters}{"20}
\DeclareMathSymbol{\omega}{\mathord}{letters}{"21}
\DeclareMathSymbol{\varepsilon}{\mathord}{letters}{"22}
\DeclareMathSymbol{\vartheta}{\mathord}{letters}{"23}
\DeclareMathSymbol{\varpi}{\mathord}{letters}{"24}
\DeclareMathSymbol{\varrho}{\mathord}{letters}{"25}
\DeclareMathSymbol{\varsigma}{\mathord}{letters}{"26}
\DeclareMathSymbol{\varphi}{\mathord}{letters}{"27}

\DeclareMathSymbol{\Gamma}{\mathord}{letters}{"00}
\DeclareMathSymbol{\Delta}{\mathord}{letters}{"01}
\DeclareMathSymbol{\Theta}{\mathord}{letters}{"02}
\DeclareMathSymbol{\Lambda}{\mathord}{letters}{"03}
\DeclareMathSymbol{\Xi}{\mathord}{letters}{"04}
\DeclareMathSymbol{\Pi}{\mathord}{letters}{"05}
\DeclareMathSymbol{\Sigma}{\mathord}{letters}{"06}
\DeclareMathSymbol{\Upsilon}{\mathord}{letters}{"07}
\DeclareMathSymbol{\Phi}{\mathord}{letters}{"08}
\DeclareMathSymbol{\Psi}{\mathord}{letters}{"09}
\DeclareMathSymbol{\Omega}{\mathord}{letters}{"0A}

\DeclareMathSymbol{\upGamma}{\mathalpha}{operators}{"00}
\DeclareMathSymbol{\upDelta}{\mathalpha}{operators}{"01}
\DeclareMathSymbol{\upTheta}{\mathalpha}{operators}{"02}
\DeclareMathSymbol{\upLambda}{\mathalpha}{operators}{"03}
\DeclareMathSymbol{\upXi}{\mathalpha}{operators}{"04}
\DeclareMathSymbol{\upPi}{\mathalpha}{operators}{"05}
\DeclareMathSymbol{\upSigma}{\mathalpha}{operators}{"06}
\DeclareMathSymbol{\upUpsilon}{\mathalpha}{operators}{"07}
\DeclareMathSymbol{\upPhi}{\mathalpha}{operators}{"08}
\DeclareMathSymbol{\upPsi}{\mathalpha}{operators}{"09}
\DeclareMathSymbol{\upOmega}{\mathalpha}{operators}{"0A}

\DeclareMathAlphabet\PazoBB{U}{fplmbb}{m}{n}%

\renewcommand{\c}[1]{\mathcal{#1}}
\newcommand{\ter}{\mathsf{T}}

\newcommand{\id}[1]{\operatorname{id}_{#1}}

\newcommand{\Set}{\mathsf{Set}}

\newcommand{\Grp}{\mathsf{Grp}}
\newcommand{\Top}{\mathsf{Top}}
\newcommand{\Alg}[2][]{\mathsf{Alg}_{#1}(#2)}

\newcommand{\sign}{\upSigma}
\newcommand{\iden}{\mathbb{I}}

\newcommand{\AS}[1][]{\mathsf{Struc}_{#1}(\sign,\iden)}
\newcommand{\ASarg}[3][]{\mathsf{Struc}_{#1}(#2,#3)}
\newcommand{\ff}[3]{U^{#1,#2}_{#3}}
\newcommand{\ffsub}[1]{U_{#1}}

\newcommand{\eval}[1][\c{C}]{e_{#1}}
\newcommand{\ind}[1]{\tilde{#1}}

\begin{document}

\title{A note on transport of algebraic structures}

\author{Henrik Holm}

\address{University of Copenhagen, 2100 Copenhagen {\O}, Denmark}
 
\email{holm@math.ku.dk}

\urladdr{http://www.math.ku.dk/\~{}holm/}





\begin{abstract}
  We study transport of algebraic structures and prove a theorem which subsumes results of Comfort and Ross on topological group structures on \smash{Stone-{\v C}ech} compactifications, of Chevalley and of Gil de Lamadrid and Jans on topological group and ring structures on universal covering spaces, and of Gleason on topological group structures on universal locally connected refinements.
\end{abstract}

\maketitle


\section{Introduction}

Transport of algebraic structures---a concept that will be made precise in \secref[Sect.~]{main}---is a well-known phenomenon. To illustrate what we have in mind, we mention some results from the literature, which have motivated this work.

\begin{prt}

\item Let $M$ be a metric space with completion \smash{$M \to \widehat{M}$}. Algebraic structures on $M$ tend to be inherited by \smash{$\widehat{M}$}. For example,  let $M=\mathbb{Q}$ be equipped with the metric induced by the Euclidian norm \smash{$|\cdot|$} or the $p$-adic norm \smash{$|\cdot|_p$}. Then one has \smash{$\widehat{M}=\mathbb{R}$} (the real numbers) or \smash{$\widehat{M}=\mathbb{Q}_p$} (the $p$-adic numbers). In either case, $M$ is a ring\footnote{\ Of course, $\mathbb{Q}$, $\mathbb{R}$ and $\mathbb{Q}_p$ are even fields ($p$ is a prime number), but a field is not an ``algebraic structure'' in the sense discussed in this paper; see~\ref{alg-struc-set}.} in which addition and multiplication are continuous functions. As it is well-known, the completion \smash{$\widehat{M}$} can, in both of these cases, also be made into a ring with continuous addition and multiplication in such a way that \smash{$M \to \widehat{M}$} becomes a ring homomorphism. This examplifies that the algebraic structure of type ``ring'' ascends\footnote{\ In this paper, we say, loosely speaking, that an algebraic structure of a given type (such as ``group'' or ``ring'') \emph{ascends}, respectively, \emph{descends}, along a map $f \colon X \to Y$ (more precisely, along an arrow in some category) if every algebraic structure of that type on $X$, respectively, on $Y$, can be ``transported'' to $Y$ (that is, in the direction of the arrow), respectively, to $X$ (that is, against the direction of the arrow), in such a way that $f \colon X \to Y$ becomes a homomorphism of algebraic structures of the type in question. The precise definitions can be found in \secref[Sect.~]{main}.} along the map \smash{$M \to \widehat{M}$}.

\item Let $X$ be a topological space with \smash{Stone-{\v C}ech} compactification $X \to \beta X$. Comfort and Ross \cite[Thm.~4.1]{ComfortRoss} showed that if $X$ is a pseudocompact topological group, then $\beta X$ admits a structure of a topological group in such a way that $X \to \beta X$ becomes a group homomorphism (and a homeomorphism onto its image). This illustrates that the algebraic structure of type ``group'' ascends along the map $X \to \beta X$ for certain types of spaces $X$.

\item Let $X$ be a topological space which has a universal covering space $\tilde{X} \to X$. Chevalley \cite[Chap.~II\S8 Prop.~5]{Chevalley} proved that if $X$ is a topological group, then $\tilde{X}$ can be made into a topological group in such a way that $\tilde{X} \to X$ becomes a group homomorphism. By similar methods, Gil de Lamadrid and Jans~\cite[Thm.~1]{GildeLamadridJans} showed that if $X$ is a topological ring, then $\tilde{X}$ can be equipped with the structure of a topological ring such that $\tilde{X} \to X$ becomes a ring homomorphism. This illustrates that the algebraic structures of type ``group'' and ``ring'' decend along the map $\tilde{X} \to X$.

\item Let $X$ be a topological space and let $X^* \to X$ be its universal locally connected refinement in the sense of Gleason~\cite[Thm.~A]{Gleason}. In Thm.~D in \emph{loc.~cit.}~it is proved that if $X$ is a topological group, then $X^*$ can be made into a topological group in such a way that the map $X^* \to X$ becomes a group homomorphism. This illustrates that the algebraic structure of type ``group'' decends along the map $X^* \to X$.

\end{prt}

The purpose of this note is to describe circumstances under which algebraic structures ascend or descend along certain types of morphisms and to give useful applications.
Our main result, \thmref{main}, is certainly not profound; it deals with an adjoint situation and its proof is completely formal. However, despite its simplicity, the result has several useful applications; some of them are collected in Theorem~A below, which subsumes the classic results (a)--(d) mentioned above. We consider these applications as the main contents of this note, and they are our justification for presenting the details that lead to \thmref{main}.

\begin{introthm*}[A]
  The following assertions hold.
  \begin{prt}
  
  \item Let $M$ be any metric space with completion \smash{$\widehat{M}$}. Every algebraic structure on $M$~ascends uniquely along the canonical map \smash{$M \to \widehat{M}$}. 
  
  \item Let $X$ be any pseudocompact and locally compact topological space with \smash{Stone-{\v C}ech} compactification $\beta X$.  Every algebraic structure on $X$ ascends uniquely along the canonical map $X \to \beta X$.
  
  \item Let $X$ be any pointed topological space which has a pointed universal covering space $\tilde{X}$. Every algebraic structure on $X$ decends uniquely along the canonical map $\tilde{X} \to X$. 
  
  \item Let $X$ be any topological space with universal locally connected refinement $X^*$. Every algebraic structure on $X$ decends uniquely along the canonical map $X^* \to X$.
  \end{prt}
\end{introthm*}

The paper is organized as follows: \secref[Sect.~]{AS} contains a few preliminaries on universal algebra, algebraic structures, and algebraic theories. In \secref[Sect.~]{main} we prove our main result and in \secref[Sect.~]{app} we apply this result in various settings and thereby give a proof of Theorem~A.

\section{Algebraic structures and algebraic theories}

\label{sec:AS}

Algebraic structures are objects like groups and rings. In general, they are sets equipped with operations subject to identities. The language of universal algebra makes this precise, and we refer to e.g.~Burris and Sankappanavar~\cite[Chap.~II\S10,11]{BS} for the relevent notions.

\begin{ipg}
  \label{alg-struc-set}
Let $\sign$ be an algebraic signature and let $\iden$ be a set of identities of type $\sign$. An \emph{algebraic structure of type $(\sign,\iden)$} (called a \emph{$(\sign,\iden)$-algebra} in Mac Lane \cite[Chap.~V\S6]{Mac}) is a $\sign$-algebra that satisfies every identity in $\iden$. A \emph{morphism} of such structures is a morphism of the underlying $\sign$-algebras. We write $\AS$ for the category of all algebraic structures of type $(\sign,\iden)$ and $\ff{\sign}{\iden}{} \colon \AS \to \Set$ for the forgetful functor.

By an \emph{algebraic structure} we just mean an algebraic structure of some type $(\sign,\iden)$.
\end{ipg}

Note that a field is not an algebraic structure in the sense above. In the literature, the category $\AS$ is often referred to as a \emph{variety} or an \emph{equational class}.

By definition, an algebraic structure of type $(\sign,\iden)$ is, in particular, a \emph{set} (and the definition of what is means for a $\sign$-algebra to satisfy an identity refers specifically to elements). In other words, \ref{alg-struc-set} defines algebraic structures in the category $\Set$. The standard way to deal with algebraic structures in more general categories goes through algebraic theories. The book \cite{ARV} by Ad{\'a}mek, Rosick{\'y}, and Vitale is an excellent account on algebraic theories, and we shall refer to this for relevent notions and results.

\begin{ipg}
\label{Alg-Set}
Following \cite[Chap.~1]{ARV} an \emph{algebraic theory} is a small category with finite products. If $\c{T}$ is an algebraic theory, then a \emph{$\c{T}$-algebra} is a functor $\c{T} \to \Set$ that preserves finite products. A \emph{morphism} of $\c{T}$-algebras is a natural transformation. The category of all $\c{T}$-algebras and their morphisms is denoted by $\Alg{\c{T}}$.
\end{ipg}

\begin{exa}
\label{exa:N}
Let $\underline{\mathbb{N}}_0$ be the category whose objects are natural numbers and zero and in which the hom-set $\underline{\mathbb{N}}_0(m,n)$ consists of all functions $\{0,\ldots,m-1\} \to \{0,\ldots,n-1\}$. This category has finite coproducts; indeed, the coproduct of objects $m,n \in \underline{\mathbb{N}}_0$ is the sum $m+n$. Thus the opposite category $\underline{\mathbb{N}}_0^\mathrm{op}$ is an algebraic theory; we denote it by $\c{N}$ as in \cite[Exa.~1.9]{ARV}. In this category, every object $n \in \c{N}$ is the $n$-fold product $n \cong 1 \times \cdots \times 1$ of the object $1$ (this also holds for $n=0$, as the empty product in a category yields the terminal object). In the cited example, it is also proved that the functor $\eval[] \colon \Alg{\c{N}} \to \Set$ given by $A \mapsto A(1)$ is an equivalence of categories. It is customary to suppress this functor.
\end{exa}

\begin{ipg}
  Let $\c{C}$ be a fixed category. Recall that a \emph{concrete category over $\c{C}$} is a pair $(\c{U},U)$ where $\c{U}$ is a category and $U \colon \c{U} \to \c{C}$ is a faithful functor. If $(\c{U},U)$ and $(\c{V},V)$ are concrete categories over $\c{C}$, then a \emph{concrete functor}
$(\c{U},U) \to (\c{V},V)$ is a functor $F \colon \c{U} \to \c{V}$ with $VF=U$. A \emph{concrete equivalence} of concrete categories $(\c{U},U)$ and $(\c{V},V)$ over $\c{C}$ is a pair of quasi-inverse concrete functors $(\c{U},U) \rightleftarrows (\c{V},V)$.
\end{ipg}

\begin{ipg}
\label{one-sorted}
Following \cite[Def.~11.3]{ARV} a \emph{one-sorted algebraic theory} is a pair $(\c{T},T)$, where $\c{T}$ is an algebraic theory whose objects are natural numbers and zero and $T \colon \c{N} \to \c{T}$ is a product preserving functor which is the identity of objects. The definition is due to Lawvere \cite{Lawvere}. 

Let $(\c{T},T)$ be a one-sorted algebraic theory. As noted above, one has $n \cong 1 \times \cdots \times 1$ ($n$ copies) in the category $\c{N}$. By applying the functor $T$ (which is the identity on objects and preserves products) to this isomorphism, one gets that $n \cong 1 \times \cdots \times 1$ also holds in $\c{T}$.
\end{ipg}

\begin{ipg}
\label{concrete}
For every one-sorted algebraic theory $(\c{T},T)$ the functor \smash{$\xymatrix@C=1.6pc{\Alg{\c{T}} \ar[r]^-{\Alg{T}} & \Alg{\c{N}} \simeq \Set}$} is faithful by \cite[Prop.~11.8]{ARV}, and hence $(\Alg{\c{T}},\Alg{T})$ is a concrete category over $\Set$. The pair $(\AS,\ff{\sign}{\iden}{})$ from \ref{alg-struc-set} is also a concrete category over $\Set$. It is proved in \cite[Thm.~13.11]{ARV} that for every algebraic signature $\sign$ and every set $\iden$ of identities of type $\sign$, there exists a one-sorted algebraic theory $(\c{T}_{\sign,\iden},T_{\sign,\iden})$ and a concrete equivalence,
\begin{equation}
  \label{eq:Struc-Set}
  (\Alg{\c{T}_{\sign,\iden}},\Alg{T_{\sign,\iden}}) \,\simeq\, (\AS,\ff{\sign}{\iden}{})\,.
\end{equation}
Moreover, by ``one-sorted algebraic duality'' \cite[Thm.~11.39]{ARV} such a one-sorted algebraic theory $(\c{T}_{\sign,\iden},T_{\sign,\iden})$ is unique up to isomorphism (of one-sorted algebraic theories).

As shown in \cite[Rmk.~11.24 and Chap.~13]{ARV}, the algebraic theory $\c{T}_{\sign,\iden}$ can be con\-structed directly from the left adjoint of the forgetful functor $\ff{\sign}{\iden}{}$. 
\end{ipg}

Now, the standard way to define algebraic structures of type $(\sign,\iden)$ in a general category (with finite products) is as follows.

\begin{dfn}
Let $\c{T}$ be an algebraic theory and let $\c{C}$ be any category with finite products. 
A \emph{$\c{T}$-algebra in $\c{C}$} is a functor $\c{T} \to \c{C}$ that preserves finite products. A \emph{morphism} of $\c{T}$-algebras in $\c{C}$ is a natural transformation. The category of all $\c{T}$-algebras in $\c{C}$ and their morphisms is denoted by $\Alg[\c{C}]{\c{T}}$. Thus, $\Alg[\Set]{\c{T}}$ is nothing but $\Alg{\c{T}}$ from \ref{Alg-Set}. 
\end{dfn}

The next lemma is straightforward to prove (the proof is the same as in the case $\c{C}=\Set$).

\begin{lem}
  Let $\c{C}$ be a category with finite products. There is an equivalence of categories $\eval \colon \Alg[\c{C}]{\c{N}} \to \c{C}$ given by $\ind{C} \mapsto \ind{C}(1)$. \qed
\end{lem}

\begin{dfn}
  \label{dfn:Struc-C}
Let $\sign$ be an algebraic signature, let $\iden$ be a set of identities of type $\sign$, and let $\c{C}$ be a category with finite products. Let $(\c{T}_{\sign,\iden},T_{\sign,\iden})$ be the unique one-sorted algebraic theory for which there is a concrete equivalence \eqref{Struc-Set}. Define the category
\begin{displaymath}
  \AS[\c{C}] := \Alg[\c{C}]{\c{T}_{\sign,\iden}}
\end{displaymath}
and the \emph{forgetful functor} \smash{$\ff{\sign}{\iden}{\c{C}}$ as the composition $\xymatrix@C=2.5pc{
     \Alg[\c{C}]{\c{T}_{\sign,\iden}} \ar[r]^-{\Alg[\c{C}]{T_{\sign,\iden}}} & \Alg[\c{C}]{\c{N}} \ar[r]^-{\eval}_-{\simeq} & \c{C}
   }$}.

We refer to an object in $\AS[\c{C}]$ as an \emph{algebraic structure of type $(\sign,\iden)$ in $\c{C}$} (even though it is not actually an object in $\c{C}$, but a product preserving functor $\c{T}_{\sign,\iden} \to \c{C}$).
\end{dfn}
      
\begin{rmk}
  \label{rmk:action}
  The action of \smash{$\ff{\sign}{\iden}{\c{C}}$} on an object $\ind{X}$ (i.e.~a functor $\c{T}_{\sign,\iden} \to \c{C}$) is $\ind{X}(1) \in \c{C}$, and the action of \smash{$\ff{\sign}{\iden}{\c{C}}$} on a morphism $\ind{\sigma} \colon \ind{X} \to \ind{Y}$ (i.e.~a natural transformation) is $\ind{\sigma}_1$. 
  
  As noted in \ref{one-sorted} one has $n \cong 1 \times \cdots \times 1$ in $\c{T}_{\sign,\iden}$. Thus, for any morphism $\ind{\sigma} \colon \ind{X} \to \ind{Y}$ of
$\c{T}_{\sign,\iden}$-algebras in $\c{C}$ one has $\ind{\sigma}_n \cong \ind{\sigma}_1 \times \cdots \times \ind{\sigma}_1$. This has two immediate consequences:
\begin{prt}
  \item The functor \smash{$\ff{\sign}{\iden}{\c{C}}$} is faithful; thus \smash{$(\AS[\c{C}],\ff{\sign}{\iden}{\c{C}})$} is a concrete category over $\c{C}$.
  \item If \smash{$\ff{\sign}{\iden}{\c{C}}(\ind{\sigma})$} is an isomorphism, then so is $\ind{\sigma}$.
\end{prt} 
\end{rmk}      

\begin{ipg}
\label{group_objects}
\emph{A priori} an algebraic strucure of type $(\sign,\iden)$ in a category $\c{C}$ is not actully an object in $\c{C}$, but instead a product preserving functor $\c{T}_{\sign,\iden} \to \c{C}$. However, it is possible to---and well-known that one can---interpret such a functor as an actual object in $\c{C}$ equipped with some additional structure. For example, if $\sign_{\Grp}$ is the algebraic signature and $\iden_{\Grp}$ is the set of identities for groups, then the concrete category $\ASarg[\c{C}]{\sign_{\Grp}}{\iden_{\Grp}}$ of algebraic structures of type $(\sign_{\Grp},\iden_{\Grp})$ in $\c{C}$ is concretely equivalent to the concrete category $\Grp(\c{C})$ of \emph{group objects} in $\c{C}$. We remind the reader that a group object in $\c{C}$ is a quadruple $(C,m,u,i)$ where $C$ is an (actual) object in $\c{C}$ and $m,u,i$ are morphisms (where $\ter$ is the terminal object in $\c{C}$):
\begin{align*}
  C \times C &\stackrel{m}{\longrightarrow} C \quad \text{(called multiplication)} \\
  \ter &\stackrel{u}{\longrightarrow} C \quad \text{(called unit)} \\ 
  C &\stackrel{i}{\longrightarrow} C \quad \text{(called inverse)}
\end{align*}  
that make the expected diagrams commutative. The category $\ASarg[\c{C}]{\sign_{\Grp}}{\iden_{\Grp}}$ is convenient for working with group structures in $\c{C}$
from a theoretical point of view, however, in specific examples (see \secref[Sect.~]{app}) it is more natural to have the category $\Grp(\c{C})$ in mind.
\end{ipg}

\section{The main result}

\label{sec:main}

Throughout this section, we fix an algebraic signature $\sign$ and a set $\iden$ of identities of type $\sign$. For a category $\c{C}$ with finite products we consider the category $\AS[\c{C}]$ of algebraic structures of type $(\sign,\iden)$ in $\c{C}$ and its forgetful functor $\ffsub{\c{C}} \colon \AS[\c{C}] \to \c{C}$ from \dfnref[Def.~]{Struc-C}. 

\begin{dfn}
By an \emph{algebraic structure of type $(\sign,\iden)$ on an object $C \in \c{C}$} we mean an object $\ind{C} \in \AS[\c{C}]$ such that $\ffsub{\c{C}}(\ind{C})=C$. 
\end{dfn}

This is the definition we shall formally use. However, as illustrated in \ref{group_objects}, one can think of an algebraic structure $\ind{C}$ on an object $C \in \c{C}$ as a pair $\ind{C}=(C,\{f_\sigma\}_{\sigma \in \sign})$ where $\{f_\sigma\}_{\sigma \in \sign}$ is a collection of morphisms in $\c{C}$, determind by the signature $\sign$, that make certain diagrams, determined by the identities $\iden$, commutative.

Now suppose that $F \colon \c{C} \to \c{D}$ is a product preserving functor between categories with finite products. There is a commutative diagram,
\begin{displaymath}
  \xymatrix@C=4pc{
    \Alg[\c{C}]{\c{T}_{\sign,\iden}} 
    \ar[d]_-{\Alg[F]{\c{T}_{\sign,\iden}}}    
    \ar[r]^-{\Alg[\c{C}]{T_{\sign,\iden}}} 
    &
    \Alg[\c{C}]{\c{N}}     
    \ar[d]^-{\Alg[F]{\c{N}}}
    \ar[r]^-{\eval}
    &
    \c{C}
    \ar[d]^-{F}
    \\
    \Alg[\c{D}]{\c{T}_{\sign,\iden}} 
    \ar[r]^-{\Alg[\c{D}]{T_{\sign,\iden}}} 
    &
    \Alg[\c{D}]{\c{N}}     
    \ar[r]^-{\eval[\c{D}]}
    &
    \c{D},
  }
\end{displaymath}
where the ``horizontal'' functors $\Alg[\c{C}]{T_{\sign,\iden}}$ and $\Alg[\c{D}]{T_{\sign,\iden}}$ map a functor $\ind{X}$ to $\ind{X} \circ T_{\sign,\iden}$, and the ``vertical'' functors $\Alg[F]{\c{T}_{\sign,\iden}}$ and $\Alg[F]{\c{N}}$ map a functor $\ind{X}$ to $F \circ \ind{X}$. If we write $\ind{F}$ for the functor $\Alg[F]{\c{T}_{\sign,\iden}}$, i.e.~$\ind{F}(\ind{X})=F\circ \ind{X}$, then the commutative diagram above is:
\begin{equation}
  \label{eq:CD}
  \begin{gathered}
  \xymatrix{
    \AS[\c{C}] \ar[r]^-{\ffsub{\c{C}}} 
    \ar[d]_-{\ind{F}}
    & \c{C} \ar[d]^-{F}
    \\
    \AS[\c{D}] \ar[r]^-{\ffsub{\c{D}}} & \c{D}.    
  }
  \end{gathered}  
\end{equation}

Suppose that $\ind{D}$ is an algebraic structure of type $(\sign,\iden)$ on a given object $D \in \c{D}$. The commutative diagram \eqref{CD} induces two functors between comma categories (\cite[Chap.~II\S6]{Mac}):
\begin{align*}
  U_F^{\ind{D}} \colon (\ind{D} \downarrow \ind{F}) \longrightarrow (D \downarrow F)
  \quad &\text{given by} \quad (\ind{C},\ind{\varphi}) \longmapsto (\ffsub{\c{C}}(\ind{C}),\ffsub{\c{D}}(\ind{\varphi})) \,, \ \ \text{and}    
  \\
  U^F_{\ind{D}} \colon (\ind{F} \downarrow \ind{D}) \longrightarrow (F \downarrow D)
  \quad &\text{given by} \quad (\ind{C},\ind{\psi}) \longmapsto (\ffsub{\c{C}}(\ind{C}),\ffsub{\c{D}}(\ind{\psi}))\,.
\end{align*}

With these functors at hand, we can now make precise what is meant by transport (ascent and descent) of algebraic structures along morphisms.

\begin{dfn}
  \label{dfn:transport}
  Let $\ind{D}$ be an algebraic structure of type $(\sign,\iden)$ on a given object $D \in \c{D}$.  
  \begin{prt}
  
    \item Let $C \in \c{C}$ be an object and let $\varphi \colon D \to F(C)$ be a morphism, i.e.~\mbox{$(C,\varphi) \in (D \downarrow F)$}. 
  
  We say that the algebraic structure $\ind{D}$ \emph{ascends along $\varphi$} if there exists an algebraic structure $\ind{C}$ of type $(\sign,\iden)$ on $C$ and a morphism $\ind{\varphi} \colon \ind{D} \to \ind{F}(\ind{C})$ in $\AS[\c{D}]$ such that $\ffsub{\c{D}}(\ind{\varphi})=\varphi$, that is, $(\ind{C},\ind{\varphi})$ is an object in $(\ind{D} \downarrow \ind{F})$ with \smash{$U_F^{\ind{D}}(\ind{C},\ind{\varphi}) = (C,\varphi)$}.
  
  We say that the algebraic structure $\ind{D}$ \emph{ascends uniquely along $\varphi$} if there a unique, up to isomorphism, object $(\ind{C},\ind{\varphi})$ in $(\ind{D} \downarrow \ind{F})$ with \smash{$U_F^{\ind{D}}(\ind{C},\ind{\varphi}) = (C,\varphi)$}.
  
    \item Let $C \in \c{C}$ be an object and let $\psi \colon F(C) \to D$ be a morphism, i.e.~\mbox{$(C,\psi) \in (F \downarrow D)$}. 
  
  We say that the algebraic structure $\ind{D}$ \emph{descends along $\psi$} if there exists an algebraic structure $\ind{C}$ of type $(\sign,\iden)$ on $C$ and a morphism $\ind{\psi} \colon \ind{F}(\ind{C}) \to \ind{D}$ in $\AS[\c{D}]$ such that $\ffsub{\c{D}}(\ind{\psi})=\psi$, that is, $(\ind{C},\ind{\psi})$ is an object in $(\ind{F} \downarrow \ind{D})$ with \smash{$U^F_{\ind{D}}(\ind{C},\ind{\psi}) = (C,\psi)$}.
  
  We say that the algebraic structure $\ind{D}$ \emph{descends uniquely along $\psi$} if there a unique, up to isomorphism, object $(\ind{C},\ind{\psi})$ in $(\ind{F} \downarrow \ind{D})$ with \smash{$U^F_{\ind{D}}(\ind{C},\ind{\psi}) = (C,\psi)$}.  

    \end{prt}
\end{dfn}

\begin{lem}
  \label{lem:adjunction}
  Let $\langle F,G,\eta,\varepsilon \rangle \colon \c{C} \to \c{D}$ be an adjunction
of product preserving functors between categories with finite products\footnote{\ Of course, $G$ always preserves products since it is a right adjoint, so the assumption on the functors is really that $F$ preserves finite products.}. The induced functors $\ind{F}$,$\ind{G}$ are part of an adjunction
\begin{displaymath}
  \langle \ind{F},\ind{G},\ind{\eta},\ind{\varepsilon} \rangle \colon \AS[\c{C}] \longrightarrow \AS[\c{D}]
\end{displaymath}  
  with $\ffsub{\c{C}}(\ind{\eta}_{\ind{C}}) = \eta_{\ffsub{\c{C}}(\ind{C})}$ for $\ind{C} \in \AS[\c{C}]$ and $\ffsub{\c{D}}(\ind{\varepsilon}_{\ind{D}}) = \varepsilon_{\ffsub{\c{D}}(\ind{D})}$ for $\ind{D} \in \AS[\c{D}]$.
\end{lem}

\begin{proof}
  We use the shorthand notation $\c{T} = \c{T}_{\sign,\iden}$ and $T = T_{\sign,\iden}$. 
  
  By definition, $\AS[\c{X}]$ is a full subcategory of the functor category $\c{X}^\c{T}$ ($\c{X}=\c{C},\c{D}$). Thus, for all product preserving functors $\ind{C} \colon \c{T} \to \c{C}$ and $\ind{D} \colon \c{T} \to \c{D}$ we must construct a natural bijection (where we have written $F\ind{C} = F \circ \ind{C}$ and $G\ind{D} = G \circ \ind{D}$):
\begin{displaymath}
  \xymatrix{
   \c{D}^\c{T}\!(F\ind{C},\ind{D}) = \c{D}^\c{T}\!(\ind{F}(\ind{C}),\ind{D})
    \ar@<0.5ex>[r]^-{\alpha} &
   \c{C}^\c{T}\!(\ind{C},\ind{G}(\ind{D})) = \c{C}^\c{T}\!(\ind{C},G\ind{D})
    \ar@<0.5ex>[l]^-{\beta}
  }
  \quad \text{(where $\beta=\alpha^{-1}$).}
\end{displaymath}  
For a natural transformation $\sigma \colon F\ind{C} \to \ind{D}$ we set $\alpha(\sigma) = G\sigma \circ \eta\mspace{1mu}\ind{C}$, and for a natural transformation $\tau \colon \ind{C} \to G\ind{D}$ we set \smash{$\beta(\tau) = \varepsilon \ind{D} \circ F\tau$}. We then have
\begin{displaymath}
\beta\alpha(\sigma) = \beta(G\sigma \circ \eta\mspace{1mu}\ind{C}) = \varepsilon \ind{D} \circ FG\sigma \circ F\eta\mspace{1mu}\ind{C} = \sigma \circ \varepsilon F\ind{C} \circ F\eta\mspace{1mu}\ind{C} = \sigma \circ \id{F\ind{C}} = \sigma\,,
\end{displaymath}  
where the first and second equalities are by definition, the third equality follows as $\varepsilon$ is a natural transformation, and the fourth equality follows as $\varepsilon F \circ F\eta = \id{F}$. Thus $\beta\,\alpha$ is the identity on \smash{$\c{D}^\c{T}\!(F\ind{C},\ind{D})$}, and a similar argument shows that $\alpha\beta$ is the identity on \smash{$\c{C}^\c{T}\!(\ind{C},G\ind{D})$}.

The unit of this adjunction is \smash{$\ind{\eta}_{\ind{C}} = \alpha(\id{\ind{F}(\ind{C})}) = \eta\mspace{1mu}\ind{C} \colon \ind{C} \to \ind{G}\ind{F}(\ind{C}) = GF\ind{C}$}. Thus,  $\ffsub{\c{C}}(\ind{\eta}_{\ind{C}})$ is the morphism $\eta_{\ind{C}(1)} \colon \ind{C}(1) \to GF\ind{C}(1)$, which is $\eta_{\ffsub{\c{C}}(\ind{C})} \colon \ffsub{\c{C}}(\ind{C}) \to GF(\ffsub{\c{C}}(\ind{C}))$; see~\rmkref{action}. A similar argument shows that  $\ffsub{\c{D}}(\ind{\varepsilon}_{\ind{D}}) = \varepsilon_{\ffsub{\c{D}}(\ind{D})}$.
\end{proof}

\begin{thm}
  \label{thm:main}
Let $\langle F,G,\eta,\varepsilon \rangle \colon \c{C} \to \c{D}$ be an adjunction
of product preserving functors between categories with finite products. Let $\sign$ be any algebraic signature and let $\iden$ be a set of identities of type $\sign$. The following conclusions hold.
  
 \begin{prt}
 \item Let $C \in \c{C}$, set $D=F(C) \in \c{D}$, and consider the unit $\eta_C \colon C \to G(D)$. Every alge\-braic structure of type $(\sign,\iden)$ on $C$ ascends uniquely along $\eta_C$. 
 
 \item Let $D \in \c{D}$, set $C=G(D) \in \c{C}$, and consider the counit $\varepsilon_D \colon F(C) \to D$. Every alge\-braic structure of type $(\sign,\iden)$ on $D$ descends uniquely along $\varepsilon_D$. 
  
 \end{prt}
\end{thm}

\begin{proof}
We only prove part \prtlbl{a} since the proof of \prtlbl{b} is similar. In the proof, we use the notation of \lemref{adjunction}. Let $\ind{C}$ be an algebraic structure of type $(\sign,\iden)$ on $C$. Then $\ind{D}=\ind{F}(\ind{C})$ is an algebraic structure of type $(\sign,\iden)$ on $D$ as $\ffsub{\c{D}}(\ind{D}) = \ffsub{\c{D}}\ind{F}(\ind{C}) = F\ffsub{\c{C}}(\ind{C}) = F(C) = D$. By \lemref{adjunction}, the unit $\ind{\eta}_{\ind{C}} \colon \ind{C} \to \ind{G}(\ind{D})$ has the property that $\ffsub{\c{C}}(\ind{\eta}_{\ind{C}}) = \eta_C$. This proves that the given algebraic structure $\ind{C}$ ascends along $\eta_C$. To prove that it ascends uniquely, let $\ind{D}_0$ be any algebraic structure of type $(\sign,\iden)$ on $D$ and let $\ind{\varphi} \colon \ind{C} \to \ind{G}(\ind{D}_0)$ be any morphism with $\ffsub{\c{C}}(\ind{\varphi}) = \eta_C$. As  $\ind{\eta}_{\ind{C}} \colon \ind{C} \to \ind{G}(\ind{D})$ is a universal arrow from $\ind{C}$ to $\ind{G}$, that is, $(\ind{D},\ind{\eta}_{\ind{C}})$ is the initial object in the comma category $(\ind{C} \downarrow \ind{G})$, there is a (unique) morphism $\ind{\delta} \colon \ind{D} \to \ind{D}_0$ which makes the left diagram in the following display commutative:
\begin{displaymath}
  \xymatrix{
    \ind{C} \ar[r]^-{\ind{\eta}_{\ind{C}}} \ar[d]_-{\ind{\varphi}} & \ind{G}(\ind{D}) \ar[dl]^-{\ind{G}(\ind{\delta})}
    \\
    \ind{G}(\ind{D}_0) & {}
  }
  \qquad \qquad
  \xymatrix{
    C \ar[r]^-{\eta_C} \ar[d]_-{\eta_C} & G(D) \ar[dl]^-{G(\ffsub{\c{D}}(\ind{\delta}))}
    \\
    G(D) & {}
  }
\end{displaymath}
The right diagram above is obtained from the left one by applying the functor $\ffsub{\c{C}}$. As the unit $\eta_C \colon C \to G(D)$ is a universal arrow from $C$ to $G$, that is, $(D,\eta_C)$ is the initial object in the comma category $(C \downarrow G)$, then the morphism $\ffsub{\c{D}}(\ind{\delta})$ must be the identity $\id{D}$ on $D$. It follows from \rmkref{action}(b) that $\ind{\delta}$ is an isomorphism.
\end{proof}

\section{Applications}

\label{sec:app}

In this final section, we apply \thmref{main} in some specific examples and thereby give a proof of Theorem~A in the Introduction.

\begin{exa}
  Let $\c{C}=\mathsf{Met}$ be the category whose objects are all metric spaces and whose morphisms are all continuous functions. This category has finite products, indeed, the product of metric spaces $(M,d_M)$ and $(N,d_N)$ is $(M \times N,d_{M \times N})$ where $d_{M \times N}$ is given by
\begin{displaymath}  
  d_{M \times N}((x_1,y_1),(x_2,y_2))=\max\{d_M(x_1,x_2),d_N(y_1,y_2)\}\,.
\end{displaymath}    
Let $\c{D}=\mathsf{CompMet}$ be the full subcategory of $\mathsf{Met}$ consisting of all complete metric spaces. Note that $\mathsf{CompMet}$ is closed under finite products in $\mathsf{Met}$ and write \mbox{$G \colon \mathsf{CompMet} \to \mathsf{Met}$} for the inclusion functor. The functor $G$ has a left adjoint, namely the functor $F$ that maps a metric space $M$ to its completion \smash{$F(M)=\widehat{M}$} (i.e. $\mathsf{CompMet}$ is a reflective subcategory of $\mathsf{Met}$). It is not hard to see that $F$ preserves finite products. The unit of the adjunction is the canonical isometry \smash{$\eta_M \colon M \to \widehat{M}$} (whose image is dense in \smash{$\widehat{M}$}).

\thmref{main}(a) applies to this setting and shows that every algebraic structure on a metric space $M$ ascends uniquely along the map \smash{$\eta_M \colon M \to \widehat{M}$}, as asserted in Theorem~A(a).
\end{exa}

\begin{exa}
  Let $\Top$ be the category of topological spaces and let $\c{D}=\mathsf{CompHaus}$ be the full subcategory hereof whose objects are all compact Hausdorff spaces. We note that $\mathsf{CompHaus}$ is closed under finite products in $\Top$ and write \mbox{$G \colon \mathsf{CompHaus} \to \Top$} for the inclusion functor. The functor $G$ has a left adjoint, namely the functor $F$ that maps a topological space $X$ to its \smash{Stone-{\v C}ech} compactification \smash{$F(X)=\beta X$}. The functor $\beta$ does not preserve finite products, for example, $\beta(\mathbb{R} \times \mathbb{R})$ is not $\beta\mathbb{R} \times \beta\mathbb{R}$; see Walker~\cite[1.67]{Walker}. 
  
  Consider therefore the full subcategory $\c{C}=\mathsf{PsLocComp}$ of $\Top$ whose objects are pseudocompact\footnote{\ A topologcal space $X$ is pseudocompact if every continuous function $X \to \mathbb{R}$ is bounded.} and locally compact\footnote{\ A topologcal space is locally compact if every point has a local base consisting of compact neighbourhoods.} spaces. A finite product of locally compact spaces is locally compact. A product of two pseudocompact spaces need not be pseudocompact\footnote{\ Gillman and Jerison \cite[9.15]{GillmanJerison} present an example, due to Nov{\'a}k and Terasaka, of a pseudocompact space $X$ for which $X \times X$ is not
pseudocompact.} but if, in addition, one of the factors is locally compact, then the product is pseudocompact by Glicksberg~\cite[Thm.~3]{Glicksberg} or \cite[8.21]{Walker}. Hence $\c{C}$ is closed under finite products in $\Top$. Every compact space is also pseudocompact, and every compact Hausdorff space is locally compact. Thus there is an inclusion $\c{D} \subset \c{C}$. Although this inclusion is ``close'' to being an equality, it is strict\footnote{\ For a pathological example of a topological space $X$ which is both pseudocompact and locally compact, but neither compact nor Hausdorff, let $X$ be any countable set with the particular point topology.}. The restriction of $F=\beta$ to $\c{C}$ does preserve finite products; this is part of proof of \cite[Thm.~3]{Glicksberg}. In conclusion, \thmref{main}(a) applies to the situation
\begin{displaymath}
  \xymatrix{
    \c{C}=\mathsf{PsLocComp} \ar@<0.5ex>[r]^-{F=\beta} & \mathsf{CompHaus}=\c{D}\,, \ar@<0.5ex>[l]^-{G}
  }
\end{displaymath}
and shows that every algebraic structure on a pseudocompact and locally compact topological space  $X$ ascends uniquely along the canonical map $\eta_X \colon X \to \beta X$ (which is the unit of the adjunction), as asserted in Theorem~A(b).
\end{exa}

\begin{exa}
  It is proved in Munkres~\cite[Cor.~82.2]{Munkres} (see also May~\cite[Chap.~3\S8]{May}) that a topological space has a universal covering space if and only if it is path connected, locally path con\-nected, and semi-locally simply connected\footnote{As (path connected) $\Rightarrow$ (connected) and since (connected)\,+\,(locally path connected) $\Rightarrow$ (path connected), the conditions (path connected)\,+\,(locally path connected) and (connected)\,+\,(locally path connected) are the~same.}. Write $\c{D}$ for the category of all such spaces, and $\c{C}=\mathsf{SConn}$ for the category of simply connected topological spaces. Note that $\c{C}$ and $\c{D}$ are closed under finite products in $\Top$; see e.g.~\cite[II\S7~Prop.~1 and II\S8~Prop.~4]{Chevalley}.

Let $\Top_*$ be the category of pointed topological spaces and denote by $\c{C}_*$ and $\c{D}_*$ the full subcategories of $\Top_*$ whose objects are the ones in $\c{C}$ and $\c{D}$, respectively. As noted above, the inclusion functor $F \colon \c{C}_* \to \c{D}_*$ preserves finite products. Since we work with pointed spaces, the universal covering space $\c{C}_* \ni \tilde{X} \to X$ of a space $X \in \c{D}_*$ has the unique mapping property, see e.g.~\cite[II\S8~Prop.~1]{Chevalley}, in other words, $F(\tilde{X})=\tilde{X} \to X$ is a universal arrow from the inclusion functor $F$ to the object $X$. By \cite[IV\S1~Thm.~2]{Mac} this means that there is a well-defined functor $G \colon \c{D}_* \to \c{C}_*$, which assigns to each $X \in \c{D}_*$ its universal covering space $G(X)=\tilde{X}$, and that this functor $G$ is right adjoint to $F$.

Thus \thmref{main}(b) applies to this setting and shows that every algebraic structure on a space $X \in \c{D}_*$ descends uniquely along the map $\varepsilon_X \colon \tilde{X} \to X$, as asserted in Theorem~A(c).
\end{exa}

\begin{exa}
  Let $\c{D} = \Top$ be the category of all topological spaces and let $\c{C} = \mathsf{LocConn}$ be the full subcategory hereof whose objects are all locally connected spaces. Note that $\c{C}$ is closed under finite products in $\c{D}$, so the inclusion functor $F \colon \mathsf{LocConn} \to \Top$ preserves finite products. The main result in Gleason~\cite{Gleason} is that $F$ has a right adjoint $G$, which to every space $X$ assigns its so-called universal locally connected refinement $G(X)=X^*$.
  
\thmref{main}(b) applies to this setting and shows that every algebraic structure on a space $X$ descends uniquely along the map $\varepsilon_X \colon X^* \to X$, as asserted in Theorem~A(d).
\end{exa}

\enlargethispage{3ex}

\section*{Acknowledgement}

We thank Anders Kock for useful comments and suggestions.

\def\cprime{$'$}
  \providecommand{\arxiv}[2][AC]{\mbox{\href{http://arxiv.org/abs/#2}{\sf
  arXiv:#2 [math.#1]}}}
  \providecommand{\oldarxiv}[2][AC]{\mbox{\href{http://arxiv.org/abs/math/#2}{\sf
  arXiv:math/#2
  [math.#1]}}}\providecommand{\MR}[1]{\mbox{\href{http://www.ams.org/mathscinet-getitem?mr=#1}{#1}}}
  \renewcommand{\MR}[1]{\mbox{\href{http://www.ams.org/mathscinet-getitem?mr=#1}{#1}}}
\providecommand{\bysame}{\leavevmode\hbox to3em{\hrulefill}\thinspace}
\providecommand{\MR}{\relax\ifhmode\unskip\space\fi MR }
\providecommand{\MRhref}[2]{%
  \href{http://www.ams.org/mathscinet-getitem?mr=#1}{#2}
}
\providecommand{\href}[2]{#2}


\begin{thebibliography}{10}

\bibitem{ARV}
Ji{\v{r}}{\'{\i}} Ad{\'a}mek, Ji{\v{r}}{\'{\i}} Rosick{\'y}, and Enrico~Maria
  Vitale, \emph{Algebraic theories}, Cambridge Tracts in Mathematics, vol. 184,
  Cambridge University Press, Cambridge, 2011, A categorical introduction to
  general algebra, With a foreword by F. W. Lawvere. \MR{MR2757312}

\bibitem{BS}
Stanley Burris and Hanamantagouda~P. Sankappanavar, \emph{A course in universal
  algebra}, Graduate Texts in Mathematics, vol.~78, Springer-Verlag, New
  York-Berlin, 1981. \MR{MR648287}

\bibitem{Chevalley}
Claude Chevalley, \emph{Theory of {L}ie {G}roups. {I}}, Princeton Mathematical
  Series, vol. 8, Princeton University Press, Princeton, N. J., 1946.
  \MR{MR0015396}

\bibitem{ComfortRoss}
William~W. Comfort and Kenneth~A. Ross, \emph{Pseudocompactness and uniform
  continuity in topological groups}, Pacific J. Math. \textbf{16} (1966),
  483--496. \MR{MR0207886}

\bibitem{GildeLamadridJans}
Jes{\'u}s Gil~de Lamadrid and James~P. Jans, \emph{Note on connectedness in
  topological rings}, Proc. Amer. Math. Soc. \textbf{8} (1957), 441--442.
  \MR{MR0087883}

\bibitem{GillmanJerison}
Leonard Gillman and Meyer Jerison, \emph{Rings of continuous functions}, The
  University Series in Higher Mathematics, D. Van Nostrand Co., Inc.,
  Princeton, N.J.-Toronto-London-New York, 1960. \MR{MR0116199}

\bibitem{Gleason}
Andrew~M. Gleason, \emph{Universal locally connected refinements}, Illinois J.
  Math. \textbf{7} (1963), 521--531. \MR{MR0164315}

\bibitem{Glicksberg}
Irving Glicksberg, \emph{Stone-\v {C}ech compactifications of products}, Trans.
  Amer. Math. Soc. \textbf{90} (1959), 369--382. \MR{MR0105667}

\bibitem{Lawvere}
F.~William Lawvere, \emph{Functorial semantics of algebraic theories and some
  algebraic problems in the context of functorial semantics of algebraic
  theories}, Repr. Theory Appl. Categ. (2004), no.~5, 1--121, Reprinted from
  Proc. Nat. Acad. Sci. U.S.A. {{\bf{5}}0} (1963), 869--872 [MR0158921] and
  {{\i}t Reports of the Midwest Category Seminar. II}, 41--61, Springer,
  Berlin, 1968 [MR0231882]. \MR{MR2118935}

\bibitem{Mac}
Saunders Mac~Lane, \emph{Categories for the working mathematician},
  Springer-Verlag, New York, 1971, Graduate Texts in Mathematics, Vol. 5.
  \MR{MR0354798}

\bibitem{May}
J.~Peter May, \emph{A concise course in algebraic topology}, Chicago Lectures
  in Mathematics, University of Chicago Press, Chicago, IL, 1999.
  \MR{MR1702278}

\bibitem{Munkres}
James~R. Munkres, \emph{Topology}, second ed., Prentice Hall, Inc., Upper
  Saddle River, NJ 07458, 2000.

\bibitem{Walker}
Russell~C. Walker, \emph{The {S}tone-\v {C}ech compactification},
  Springer-Verlag, New York-Berlin, 1974, Ergebnisse der Mathematik und ihrer
  Grenzgebiete, Band 83. \MR{MR0380698}

\end{thebibliography}
\end{document}